\documentclass{amsart}

\usepackage{amsmath,
            amssymb,
            bibentry,
            enumitem,
            graphicx,
            mathrsfs,
            tabu,
            tikz,cite}
\usepackage[comma,numbers,sort,square]{natbib}
\usepackage[colorlinks=true,
            linktocpage=true,
            linkcolor=magenta,
            citecolor=magenta,
            urlcolor=magenta]{hyperref}
\PassOptionsToPackage{hyphens}{url}
\usetikzlibrary{arrows,shapes,snakes}

\newtheorem{theorem}{Theorem}[section]
\newtheorem{proposition}[theorem]{Proposition}

\theoremstyle{definition}
\newtheorem{definition}[theorem]{Definition}

\newcommand{\seq}[1]{\langle #1 \rangle}

\newcommand{\RCA}{\mathsf{RCA}}
\newcommand{\ACA}{\mathsf{ACA}}

\begin{document}

\title{A note on the reverse mathematics of the sorites}

\author{Damir D. Dzhafarov}
\address{Department of Mathematics\\
University of Connecticut\\
Storrs, Connecticut U.S.A.}
\email{damir@math.uconn.edu}


\thanks{The author was partially supported by NSF Grant DMS-1400267. The author is grateful to E.~N.~Dzhafarov for helpful discussions. }

\maketitle

\begin{abstract}
Sorites is an ancient piece of paradoxical reasoning
pertaining to sets with the following properties: (Supervenience)
elements of the set are mapped into some set of ``attributes'';
(Tolerance) if an element has a given attribute then so are the elements
in some vicinity of this element; and (Connectedness) such vicinities
can be arranged into pairwise overlapping finite chains connecting
two elements with different attributes. Obviously, if Superveneince is assumed, then (1) Tolerance implies lack of Connectedness, and (2) Connectedness implies lack of Tolerance. Using a very general but precise definition of ``vicinity'', Dzhafarov and Dzhafarov (2010) offered two formalizations of these mutual contrapositions. Mathematically, the formalizations are equally valid, but in this paper, we offer a different basis by which to compare them. Namely, we show that the formalizations have different proof-theoretic strengths when measured in the framework of reverse mathematics: the formalization of (1) is provable in $\mathsf{RCA}_0$, while the formalization of (2) is equivalent to $\ACA_0$ over $\RCA_0$. Thus, in a certain precise sense, the approach of (1) is more constructive than that of (2).
\end{abstract}

\section{Introduction}

In November of 2009, the {\em Reverse Mathematics: Foundations and Applications} workshop at the University of Chicago asked about using mathematical logic as a possible new basis for judging and comparing alternative and competing quantitative approaches to problems in cognitive science. There have been several papers written in this direction (e.g., \cite{Marcone-2007}, \cite{Dzhafarov-2011a}), and this note is a further such contribution. In it, we show that two ostensibly equivalent mathematical approaches for a certain problem in cognitive science can be teased apart in terms of their logical (or proof-theoretic) strength.


Sorites is a ``paradox'' attributed to Eubulites
of Miletus, a philosopher of the Megarian School in the 4th century
BCE. It continues to be of interest to philosophers of mind and
cognitive scientists. The essence of the two original versions of Sorites (known as The
Heap and The Bald Man) is that if natural numbers can be classified as very
large and not very large, and if $N$ is a very large number (a heap
of grains, a full head of hair), then so is $N-1$; but by repeated
subtractions of 1 (removing grains or hairs one-by-one) one can get from $N$
to any $n$ that is not a very large number. There are obvious analogues of Sorites in a continuum of real numbers (e.g., growing a very short person by a sufficiently small amount would not make this person not very short), and spaces without linear orders (e.g., in the set of spectra of color patches,
a sufficiently small change in the spectral composition of a patch
judged to be red would not change its redness). There is an opinion
that due to the use of the notion of ``small changes'' (sufficiently small, or as small as possible) Sorites requires a metric space \cite{Williamson-1994} or at least a full-fledged topological
space \cite{WC-2010}. However, it has been shown in \cite{DD-2010a} that the
most general formulation of Sorites only requires a variant of the
pre-topological structure proposed by Fr\'{e}chet and dubbed by him \emph{V-spaces} (see, e.g., \cite{Sierpinski-1952}).

\begin{definition}\label{defn:Vspace}
	A \emph{V-space} is a pair $(X,\{\mathcal{V}_x : x \in X\})$, where $X$ is a set, and for each $x \in X$, $\mathcal{V}_x$ is a non-empty collection of subsets of $X$ containing $x$.
\end{definition}

When we have fixed a particular V-space $(X,\{\mathcal{V}_x : x \in X\})$, we call each $V \in \mathcal{V}_x$ a \emph{vicinity} of $x$. Each $x \in X$ has at least one vicinity, and one can think of each such vicinity as representing a ``sense'' in which the elements of that vicinity are close to $x$. Since $x$ belongs to each of its vicinities, it is therefore close to itself in every sense. On the other hand, $\{ x\} $ need not belong to $\mathcal{V}_{x}$,
and more generally, if $V_{x}\in\mathcal{V}_{x}$ and $V'\subset V_{x}$,
$V$ need not belong to $\mathcal{V}_{x}$. Furthermore, if some $y \in X$ belongs to some vicinity of $x$, it need not be the case that $x$ belongs to some vicinity of $y$. In other words, $y$ can be close to $x$ in some sense, without $x$ needing to be close to $y$ in any sense; the notion of ``being close to in some sense'' is not necessarily symmetric. The vicinities of a V-space can be used for the following natural definition of connectedness.

\begin{definition}\label{D:cover}
	Let $(X,\{\mathcal{V}_x : x \in X\})$ be a V-space.
	\begin{enumerate}
		\item A \emph{cover} of this V-space is a sequence $\{V_x : x \in X\}$ such that $V_x \in \mathcal{V}_x$ for each $x \in X$.
		\item Two elements $a,b \in X$ are \emph{connected} in this V-space if for every cover $\{V_x : x \in X\}$ there is a finite sequence $x_0,\ldots,x_k$ of elements of $X$ such that $a = x_0$, $b = x_k$, and $V_{x_j} \cap V_{x_{j+1}} \neq \emptyset$ for each $j < k$.
		\item  If $a,b \in X$ are not connected, then we say a cover $\{V_x : x \in X\}$ for which there is no finite sequence $x_0,\ldots,x_k$ as in (2) \emph{witnesses} that $a$ and $b$ are not connected.
	\end{enumerate}
\end{definition}

Using the language of V-spaces, Dzhafarov and Dzhafarov \cite{DD-2010a} formulated the following theorem central to our analysis.

\begin{theorem}[{\cite[Theorem 3.5]{DD-2010a}}]\label{thm:nontol}
	Let $(X,(\mathcal{V}_x : x \in X))$ be a $V$-space, $Y$ a set, and $\pi : X \to Y$ a function. If $a,b \in X$ are connected and $\pi(a) \neq \pi(b)$, then there exists a $x \in X$ such that $\pi$ is not constant on any vicinity of $x$.
\end{theorem}

The significance of this theorem for Sorites is as follows. The soritical reasoning is based on the assumption that the function $\pi : X \to Y$ can be chosen so that every $x \in X$ has a vicinity $V_x$ with $\pi(y)=\pi(x)$ for any $y \in V_x$. This assumption is called Tolerance. The assumption that there are connected $a,b \in X$ with $\pi(a) \neq \pi(b)$ is called Connectedness. Using this terminology, the theorem above says that, for any V-space and any function $\pi$, Connectedness implies lack of Tolerance.

Dzhafarov and Dzhafarov \cite{DD-2010a} also considered the following alternative formulation, in which tolerance is assumed a priori. The theorem below can, by analogy with Theorem \ref{thm:nontol}, be seen as a formalization of the statement that Tolerance implies lack of Connectedness.

\begin{definition}
	Let $X$ and $Y$ be sets, and $\pi: X \to Y$ a function. The \emph{V-space induced by $\pi$} is $(X,\{\mathcal{V}_x : x \in X\})$, where for each $x \in X$, $\mathcal{V}_x$ contains the single vicinity $\{y \in X: \pi(x) = \pi(y)\}$.
\end{definition}

\begin{theorem}[{\cite[Theorem 3.10]{DD-2010a}}]\label{thm:nonconn}
	Let $X$ and $Y$ be sets and $\pi : X \to Y$ a function. If $\pi(a) \neq \pi(b)$ for some $a,b \in X$, then $a$ and $b$ are not connected in the V-space induced by $\pi$.
\end{theorem}
 

It should be noted for completeness, that in many specific versions of Sorites the main culprit of the ensuing contradiction is the very assumption that there is a function mapping a specific set $X$, such as the set of heights, into a set of attributes $Y$, such as ``tall, not tall." However, this assumption (called Supervenience) often can be saved or at least made plausible by redefining the set $X$ (e.g., replacing its elements by sequences in which they are listed) or the set of the attributes $Y$ (e.g., replacing such attributes as ``tall, not tall" with probability distributions thereof). With this in mind, we can say that the two theorems above tell us that Supervenience can only be achieved by dispensing with either Tolerance or Connectedness. Which of them it is in specific cases can be revealed by adopting the ``behavioral" approach to Sorites \cite{DD-2010a, DD-TA}. We do not need to discuss this approach here: our focus in the present paper is the proof-theoretic strength of Theorems \ref{thm:nontol} and \ref{thm:nonconn}.



\section{Reverse mathematical analysis}

Reverse mathematics is an area of mathematical logic devoted to classifying mathematical theorems according to their proof theoretic strength.  The goal is to calibrate this strength according to how much comprehension is needed to establish the existence of the sets needed to prove the theorem (i.e., according to how complex the formulas specifying such sets must be).  This is a two-step process. The first involves searching for a comprehension scheme sufficient to prove the theorem, while the second gives sharpness by showing that the theorem is in fact equivalent to this comprehension scheme over some base (minimal) theory. In practice, we use for these comprehension schemes certain subsystems of second-order arithmetic.  As our base theory, we use a weak subsystem called $\RCA_0$, which suffices to prove the existence of the computable sets, but not of any non-computable ones. As such, $\RCA_0$ corresponds to computable or constructive mathematics.  A strictly stronger system is $\ACA_0$, which adds to $\RCA_0$ comprehension for sets described by arithmetical formulas. $\ACA_0$ is considerably stronger than $\RCA_0$, sufficing to prove the existence of, e.g., the halting set, and many other non-computable sets. There is, more generally, a rich and fruitful relationship between reverse mathematics on the one hand, and computability theory on the other (see, e.g., \cite{Shore-2010} for a discussion). We refer the reader to Simpson \cite{Simpson-2009} or Hirschfeldt \cite{Hirschfeldt-2014} for background on reverse mathematics, and to Soare \cite{Soare-2016} for background on computability.

In this section, we provide the computability-theoretic and reverse mathematical analysis of Theorems \ref{thm:nontol} and \ref{thm:nonconn}. We begin by formalizing the concepts from Definition \ref{defn:Vspace} in a countable setting.

\begin{definition}\label{defn:vicinities}
	Let $X$ be a non-empty subset of $\omega$.
	\begin{enumerate}
		\item A \emph{weak system of vicinities for $x$ in $X$} is a sequence $\mathcal{W} = \seq{W_n : n \in \omega}$ such that $x \in W_n \subseteq X$ for all $n$.
		\item A \emph{weak V-space} is a pair $(X,\{\mathcal{W}_x : x \in X\})$, where for each $x \in X$, $\mathcal{W}_x$ is a weak system of vicinities for $x$ in $X$.
		\item A \emph{strong system of vicinities for $x$ in $X$} is a sequence $\mathcal{S} = \seq{S_n : n \in I}$, where $I$ is a non-empty (possibly finite) initial segment of $\omega$, $x \in S_n \subseteq X$ for all $n \in I$, and $S_n \neq S_m$ for all $n,m \in I$ with $n \neq m$.
		\item A \emph{strong V-space} is a pair $(X,\{\mathcal{S}_x : x \in X\})$, where for each $x \in X$, $\mathcal{S}_x$ is a strong system of vicinities for $x$ in $X$.
	\end{enumerate}
\end{definition}

Note that every strong system of vicinities for $x$ in $X$ computes a weak such system. Namely, if $\seq{S_n : n \in I}$ is a strong system of vicinities for $x$, define a weak system of vicinities $\seq{W_n : n \in \omega}$ for $x$ by setting $W_n = S_n$ for all $n \in I$, and $W_n = S_0$ for all $n \notin I$. The converse is false, because in a weak system of vicinities $\seq{W_n : n \in \omega}$ it could in principle be that $W_n = W_m$ for some $n \neq m$, and there is no computable way to tell when this is the case.

\begin{proposition}\label{prop:0primeupper}
	Let $(X,\{\mathcal{W}_x : x \in X\})$ be a computable weak V-space, $Y$ a computable set, and $\pi : X \to Y$ a computable function. If there exist $a,b \in X$ with $\pi(a) \neq \pi(b)$, but every $x \in X$ has a vicinity on which $\pi$ is constant, then there is a $\emptyset'$-computable cover witnessing that $a$ and $b$ are not connected.
\end{proposition}

\begin{proof}
	For each $x \in X$, write $\mathcal{W}_x = \seq{W_{x,n} : n \in \omega}$. Now given $x$, search computably in $\emptyset'$ for the least $n$ such that $\pi(y) = \pi(x)$ for all $y \in W_{x,n}$, which exists by assumption, and define $V_x = W_{x,n}$. Then $\seq{V_x : x \in X}$ is a cover of $X$, and $\pi$ is constant on each $V_x$. Since $\pi(a) \neq \pi(b)$, this cover obviously witnesses that $a$ and $b$ are not connected.
\end{proof}

In the proof of the following result, we fix an computable enumeration $\emptyset'[s]$ of $\emptyset'$. So for all $x$, we have $x \in \emptyset'$ if and only if $x \in \emptyset'[s]$ for some $s$, in which case also $x \in \emptyset'[t]$ for all $t \geq s$. We write $x \searrow \emptyset'[s]$ if $x \in \emptyset'[s]$ and $x \notin \emptyset'[t]$ for any $t < s$.

\begin{proposition}\label{prop:0primelower}
	There exists a computable strong V-space $(X,\{\mathcal{S}_x : x \in X\})$, a computable function $\pi : X \to \{0,1\}$, and $a,b \in X$ with the following properties:
	\begin{enumerate}
		\item $\pi(a) \neq \pi(b)$;
		\item every $x \in X$ has an vicinity on which $\pi$ is constant;
		\item every cover witnessing that $a$ and $b$ are not connected computes $\emptyset'$.
	\end{enumerate}
\end{proposition}

\begin{proof}
	We work with $X = \omega$. Let $a < b$ be any two numbers not in $\emptyset'$. For every $x \in X$, we uniformly construct a computable strong system of vicinities. Each of $a$ and $b$ will have a single vicinity, $V_a$ and $V_b$, respectively, while every other $x$ will have infinitely many vicinities, $V_{x,0}, V_{x,1}, \ldots$. Specifically, we let
	\[
		V_a = \{a\} \cup \{ \seq{x,s} \in \omega^2 : x \searrow \emptyset'[s] \},
	\]
	and
	\[
		V_b = \{b\} \cup \{ \seq{x,s} \in \omega^2 : s \geq 1 \wedge x \searrow \emptyset'[s-1] \},
	\]
	and for every $x \notin \{a,b\}$ and for every $n$, we let
	\[
		V_{x,n} = \{x\} \cup \{ \seq{x,t} \in \omega^2 : t \geq n\}.
	\]
	Obviously, each of these vicinities is such that the resulting V-space is computable. Note that if $x$ is different from $a$ and $b$ then $V_{x,n} \neq V_{x,m}$ for all $n \neq m$. Also, if $x$ is enumerated into $\emptyset'$ at some stage $s$, then $\seq{x,s}$ belongs to $V_a$ and $\seq{x,s+1}$ to $V_b$, so for any $n \leq s$, $V_a \cap V_{x,n} \neq \emptyset$ and $V_b \cap V_{x,n} \neq \emptyset$. It follows that, for any such $n$, no cover containing $V_a$, $V_b$, and $V_{x,n}$ can witness that $a$ and $b$ are not connected. We shall make use of this fact below.
	
	We next define the function $\pi : X \to \{0,1\}$. To begin, set $\pi(a) = 0$ and $\pi(b) = 1$. For all other numbers, we proceed by induction. Unless already defined, set $\pi(0) = 0$ and $\pi(1) = 0$. Now fix $y > 1$ such that $\pi(y)$ is undefined, and assume we have defined $\pi$ on all smaller numbers. Say $y = \seq{x,s}$, so that in particular $x < y$. If $x \searrow \emptyset'[s]$, set $\pi(y) = 0$; if $s \geq 1$ and $x \searrow \emptyset'[s-1]$, set $\pi(y) = 1$; and otherwise, set $\pi(y) = \pi(x)$. This completes the definition. It is immediate that $\pi$ is constant on each of $V_a$ and $V_b$. For every $x$ different from $a$ and $b$, if $x \notin \emptyset'$ then $\pi$ is also constant on each $V_{x,n}$. And if $x$ is enumerated into $\emptyset'$, say at stage $s$, then $x$ is constant on every $V_{x,n}$ for $n \geq s+2$.
	
	It follows by construction that our V-space and function $\pi$ satisfy conditions (1) and (2) in the statement of the proposition. We conclude by verifying property (3). Fix any cover witnessing that $a$ and $b$ are not connected. Since $a$ and $b$ each have just one vicinity, the cover must contain $V_a$ and $V_b$. As noted above, if $x \in \emptyset'$, then this cover cannot contain $V_{x,n}$ for any $n \leq s$, where $s$ is the stage at which $x$ is enumerated into $\emptyset'$. It follows that $x \in \emptyset'$ if and only if $x$ is different from $a$ and $b$, and $x \in \emptyset'[t]$ for the least $t$ such that $\seq{x,t}$ belongs to the vicinity of $x$ in the cover. Thus, $\emptyset'$ is computable from the cover, as desired.
\end{proof}

Combining the above results allows us to characterize the proof-theoretic strength of Theorem \ref{thm:nontol}.

\begin{theorem}
	The following are equivalent over $\RCA_0$.
	\begin{enumerate}
		\item $\ACA_0$.
		\item Let $(X,\{\mathcal{W}_x : x \in X\})$ be a weak V-space, $Y$ a set, and $\pi : X \to Y$ a function. Suppose $a,b \in X$ are connected in this V-space, and $\pi(a) \neq \pi(b)$. Then there exists an $x \in X$ such that $\pi$ is not constant on any $W \in \mathcal{W}_x$.
		\item Let $(X,\{\mathcal{S}_x : x \in X\})$ be a strong V-space, $Y$ a set, and $\pi : X \to Y$ a function. Suppose $a,b \in X$ are connected in this V-space, and $\pi(a) \neq \pi(b)$. Then there exists an $x \in X$ such that $\pi$ is not constant on any $S \in \mathcal{S}_x$.
	\end{enumerate}
\end{theorem}

\begin{proof}
	The implication from part 1 to part 2 follows by formalizing the proof of Proposition \ref{prop:0primeupper} in $\ACA_0$. The implication from 2 to 3 is immediate by the remark following Definition \ref{defn:vicinities}. The implication 3 to 1 follows by formalizing the proof of Proposition \ref{prop:0primelower} in $\RCA_0$.
\end{proof}

We obtain a very contrasting result concerning the strength of Theorem \ref{thm:nonconn}. It is an easy observation that if $X$, $Y$, and $\pi : X \to Y$ are all computable, then the V-space induced by $\pi$ is a computable strong V-space. Formalizing this, we have that, given sets $X$ and $Y$ and a function $\pi : X \to Y$, $\RCA_0$ can prove the existence of the V-space (as a storng V-space) induced by $\pi$.

\begin{theorem}
	$\RCA_0$ proves the following statement. Let $X$ and $Y$ be sets, and $\pi : X \to Y$ a function with $\pi(a) \neq \pi(b)$ for some $a,b \in X$. Then $a$ and $b$ are not connected in the V-space induced by $\pi$.
\end{theorem}

\begin{proof}
	We argue in $\RCA_0$. The only cover in the V-space induced by $\pi$ is $\seq{S_x : x \in X}$, where $S_x$ is the (unique) vicinity of $x$, $\{y \in X : \pi(x) = \pi(y)\}$. Suppose $\seq{x_0,\ldots,x_k}$ is a finite sequence of elements of $X$ such that $a = x_0$, $b = x_k$, and $S_{x_j} \cap S_{x_{j+1}} \neq \emptyset$ for each $j < k$. Define $\seq{y_j : j < k}$ such that $y_j$ is the least element of $S_{x_j} \cap S_{x_{j+1}}$ for all $j < k$, which exists by $\Delta^0_1$ comprehension, and is well-defined by assumption. Thus, $\pi(y_j) = \pi(y_{j+1})$ by assumption, so $\pi(y_j) = \pi(a)$ for all $j \leq k$ by $\Sigma^0_1$ induction. But then $\pi(a) = \pi(y_{k-1}) = \pi(b)$, a contradiction.
\end{proof}

\section{Conclusion}

In the analysis of Sorites, Supervenience implies incompatibility of Connectedness and Tolerance. We have shown that the two implications forming this incompatibility, Connectedness $\rightarrow$ $\neg$Tolerance and Tolerence $\rightarrow$ $\neg$Connectedness, when formalized using the general framework of Fr\'{e}chet spaces as in Dzhafarov and Dzhafarov \cite{DD-2010a}, have different proof-theoretic strength: the formalization of the former implication has the strength of $\RCA_0$, while the formalization of the latter has the strength of $\ACA_0$. In this sense, the implication Connectedness $\rightarrow$ $\neg$Tolerance can be formalized by more constructive methods than its contrapositive.

\end{document}